\documentclass[reqno]{amsart}
\usepackage{amssymb,upgreek,graphicx,graphics,epsf,colordvi,float}

\usepackage{pdfsync}
\usepackage[usenames]{color}

\numberwithin{equation}{section}

\newtheorem{theorem}{Theorem}[section]
\newtheorem{proposition}[theorem]{Proposition}

\theoremstyle{definition}

\theoremstyle{definition} 

\newcommand{\bea}{\begin{eqnarray}}
\newcommand{\eea}{\end{eqnarray}}
\newcommand{\beas}{\begin{eqnarray*}}
\newcommand{\eeas}{\end{eqnarray*}}
\newcommand{\beq}{\begin{equation}}
\newcommand{\eeq}{\end{equation}}

\def\ep{\varepsilon}

\newcommand{\cA}{\mathcal A}
\newcommand{\cB}{\mathcal B}

\newcommand{\cE}{\mathcal E}
\newcommand{\cF}{\mathcal F}

\newcommand{\cN}{\mathcal N}

\newcommand{\cMH}{\mathcal{MH}}
\newcommand{\cSM}{\mathcal{SM}}

\newcommand{\C}{\mathbb C}
\newcommand{\R}{\mathbb R}
\newcommand{\RR}{\mathbb R}

\newcommand{\lj}{[\![}
\newcommand{\rj}{]\!]}

\newcommand{\st}{\,;\,}

\DeclareMathSymbol{\complement}{\mathord}{AMSa}{"7B}

\def\vv<#1>{\langle #1\rangle}
\def\Vv<#1>{\bigl\langle #1\bigr\rangle}

\begin{document}

\frenchspacing

\title[On the Muskat problem]
{On the Muskat problem}

\author[J.~Pr\"uss]{Jan Pr\"uss}
\address{Institut f\"ur Mathematik \\
         Martin-Luther-Universit\"at Halle-Witten\-berg\\
         Theodor-Lieser-Strasse 5\\
         D-60120 Halle, Germany}
\email{jan.pruess@mathematik.uni-halle.de}

\author[G.~Simonett]{Gieri Simonett}
\address{Department of Mathematics\\
         Vanderbilt University \\
         Nashville, TN~37240, USA}
\email{gieri.simonett@vanderbilt.edu}

\thanks{The research of G.S.\ was partially
supported NSF (DMS-0600870).}

\subjclass[2010]{Primary: 35R35, 25R37, 35B35, 35K55, 35Q35, 76E17; Secondary: 76S05, 80A22}
 \keywords{Muskat problem, free boundary problem, porous medium, Darcy's law, phase transition, 
 Lyapunov function, normally stable, normally hyperbolic}

\begin{abstract}
Of concern is the motion of two fluids separated by a free interface in a porous medium,
where the velocities are given by Darcy's law.
We consider the case with and without phase transition.
It is shown that the resulting models can be understood as purely geometric evolution laws,
where the motion of the separating interface depends in a non-local way on the mean curvature.
It turns out that the models are volume preserving and surface area reducing,
the latter property giving rise to a Lyapunov function. 
We show well-posedness of the models, characterize all equilibria,
and study the dynamic stability of the equilibria.
Lastly, we show that solutions which do not develop singularities exist globally
and converge exponentially fast to an equilibrium.
\end{abstract}
\maketitle

\vspace{-.5cm}
\section{Introduction}
The Muskat flow models the evolution of the interface between two  fluids in a porous medium
and  was introduced by Muskat~\cite{Mus34} in 1934, see also~\cite{MuWy37}. 

Suppose that two  fluids, fluid$_1$ and fluid$_2$, occupy the bounded regions $\Omega_1(t)$ and $\Omega_2(t)$ in $\R^n$
such that $\Omega_i(t)\subset\Omega$ and
$\bar\Omega_1(t)\cup\bar\Omega_2(t)=\bar\Omega$. 
Let $\Gamma(t)=\partial\Omega_1(t)$ denote the interface separating the fluids.
In the following we assume that $\Omega_2(t)$, called the continuous phase, is in contact with $\partial\Omega$,
while $\Omega_1(t)$, the disperse phase, is not.
Moreover, $\nu_\Gamma=\nu_\Gamma(t,\cdot)$ 
denotes the unit normal field on $\Gamma(t)$, pointing into $\Omega_2(t)$,
see Figure 1 for the geometric setting. 

Let $u_i$ be the velocity, $\pi_i$ the pressure, $\varrho_i$ the density, 
and $\mu_i$ the viscosity of fluid$_i$, respectively.
Moreover, let $u_\Gamma$ denote the velocity of $\Gamma=\{\Gamma(t):t\ge 0\}$ and $V_\Gamma:=u_\Gamma\cdot\nu_\Gamma$ 
the corresponding normal velocity (in the direction of $\nu_\Gamma$).
If there are no sources of mass in the bulk then conservation of mass is given by the {\em continuity equation}
\begin{equation*}
\label{consmass}
\partial_t \varrho_i + {\rm div}\, (\varrho_i u_i)=0\quad  \mbox{in}\;\;\Omega_i(t).
\end{equation*}
If there is no surface mass on $\Gamma(t)$, we also have the jump condition
\begin{equation}
\label{consmassif}
[\![\varrho(u-u_\Gamma)\cdot\nu_\Gamma]\!] =0 \quad \mbox{on}\;\; \Gamma(t),
\end{equation}
where $[\![\phi]\!]=\phi_2|_{\Gamma(t)}-\phi_1|_{\Gamma(t)}$ denotes the jump of the continuous quantity $\phi$, 
defined on $\Omega_1(t)\cup\Omega_2(t)$, across $\Gamma(t)$.
\begin{figure}
\centering
\includegraphics[width=7cm, height=5cm]{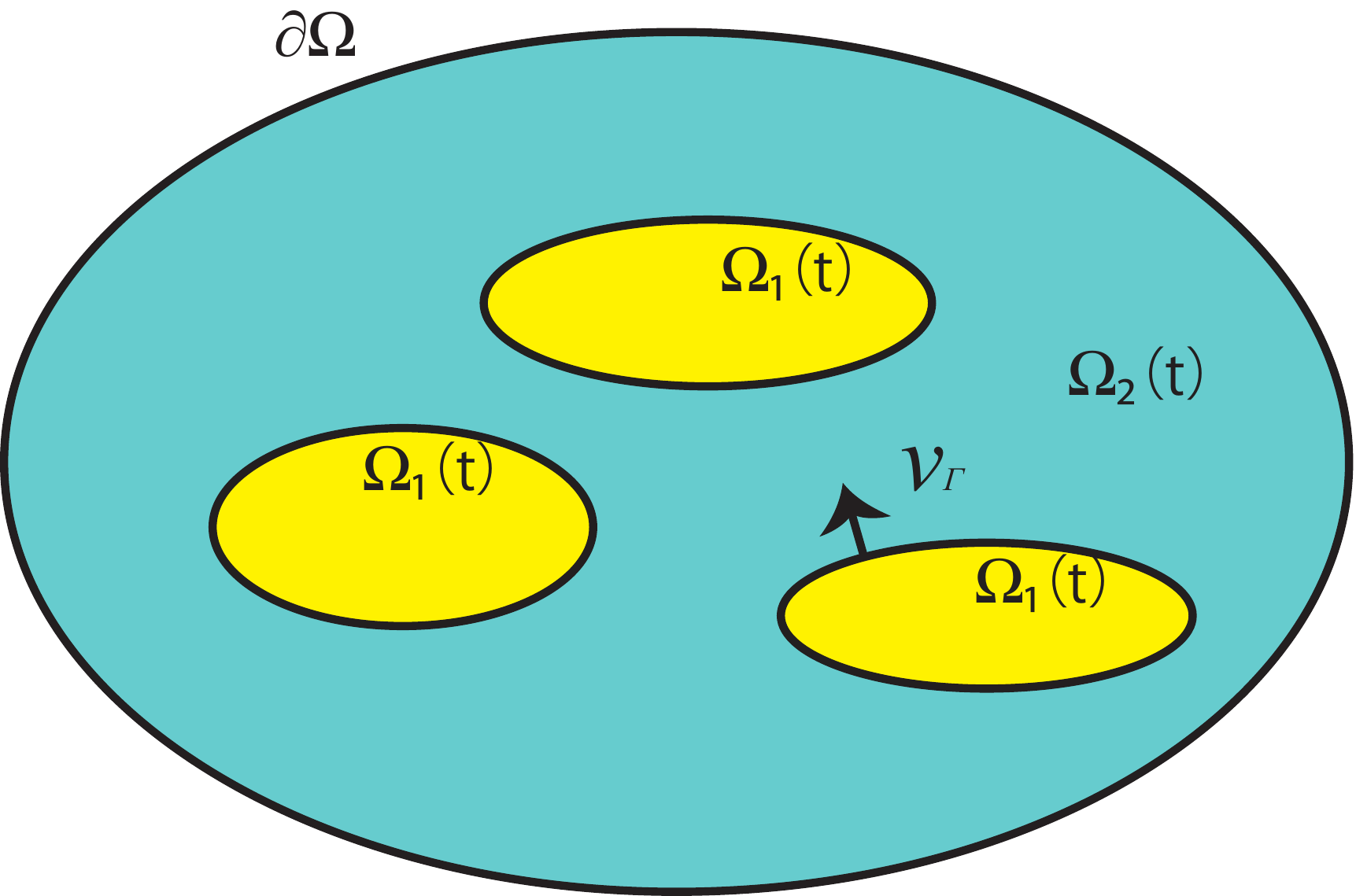}
\caption{A typical geometry}
\end{figure}
The interfacial mass flux $j_\Gamma$,
{\em phase flux} for short, is defined by means of
\begin{align}
\label{phaseflux}
 j_\Gamma:= \varrho(u-u_\Gamma)\cdot\nu_\Gamma.
 \end{align}
We note that $j_\Gamma$ is well-defined, as \eqref{consmassif} shows.
If $j_\Gamma\equiv 0$ then 
\begin{equation*}
\label{V-no transition}
V_\Gamma = u_\Gamma\cdot\nu_\Gamma = u_i\cdot\nu_\Gamma,
\end{equation*}
and in this case, the interface $\Gamma(t)$ is advected with the velocity field $u$.
On the other hand, if $j_\Gamma\not\equiv0$, {\em phase transition} occurs, and 
the normal velocity can then be expressed as
\begin{equation*}
\label{V-transtion}
\lj\varrho\rj V_\Gamma= [\![\varrho u\cdot \nu_\Gamma]\!].
\end{equation*}
In this case we will always assume that $\varrho_1\neq \varrho_2$.
In the following, we will only consider the completely incompressible case where
$\varrho_i>0$ is constant.

Modeling flows in porous media often relies on {\em Darcy's law}, which reads
\begin{equation}
\label{Darcy}
 u_i = -\frac{\kappa}{\mu_i} \nabla\pi_i, \quad i=1,2,
\end{equation}
where $\kappa>0$ is the permeability of the porous medium;
to shorten notation, we set $k_i={\kappa}/{\mu_i}$, $i=1,2$.

If no phase transition takes place we obtain from Darcy's law
\begin{equation}
\label{no-jump}
 [\![u\cdot\nu_\Gamma]\!]= -[\![k\partial_\nu \pi]\!]=0,
\end{equation}
and the normal velocity is then given by
\begin{equation}
\label{V-Gamma}
V_\Gamma = u\cdot\nu_\Gamma = - k\partial_\nu \pi.
\end{equation}
By \eqref{no-jump}, the right hand side of \eqref{V-Gamma} does not depend on the phases, and hence 
the expression for $V_\Gamma$ is  unambiguous.

In case of phase transition, the normal velocity is given by
\begin{equation}
\label{V-tansition}
\lj \varrho\rj V_\Gamma =\lj \varrho u\cdot \nu_\Gamma \rj= -\lj \varrho k\partial_\nu \pi\rj.
\end{equation}

Finally, we assume that the capillary pressure $\pi_c=[\![\pi]\!]$ is given by
\begin{equation}
\label{pressure}
[\![\pi]\!] =\sigma H_\Gamma,
\end{equation}
where $H_\Gamma =-{\rm div}_{\Gamma}\nu_\Gamma$ denotes the $(n-1)$-fold mean curvature, that is, the sum
of the principal curvatures of $\Gamma(t)$, and $\sigma>0$ is the surface tension. 
Here the convention is that $H_\Gamma= -(n-1)/R$ for a sphere of radius $R$ in $\R^n$.

The resulting problem in the case without phase transition is the well-known 
  {\em Muskat problem}, or  {\em Muskat flow}, 
which is given by
\begin{equation}
\label{Mu}
\begin{aligned}
 \Delta\pi&=0
&&\mbox{in} &&\Omega\setminus \Gamma(t),\\
\partial_\nu \pi & =0  &&\mbox{on} &&\partial\Omega,\\
[\![\pi]\!] &=\sigma H_\Gamma &&\mbox{on} &&\Gamma(t),\\
[\![k\partial_\nu \pi]\!]&=0 &&\mbox{on} &&\Gamma(t),\\
V_\Gamma &= -k\partial_\nu\pi  &&\mbox{on} && \Gamma(t),\\
\Gamma(0)&=\Gamma_0.  && &&
\end{aligned}
\end{equation}
If $j_\Gamma\neq 0$ and $\varrho_1\neq \varrho_2$, 
we obtain
the {\em Muskat flow with phase transition}
\begin{equation}
\label{MuT}
\begin{aligned}
 \Delta \pi &=0
&&\mbox{in} && \Omega\setminus \Gamma(t), \\
\partial_\nu \pi & =0 && \mbox{on} && \partial\Omega, \\
[\![\pi/\varrho]\!]&=0 &&\mbox{on} && \Gamma(t),\\
[\![\pi]\!] &=\sigma H_\Gamma &&\mbox{on} && \Gamma(t), \\
[\![\varrho]\!]V_\Gamma &=-[\![\varrho k\partial_\nu\pi]\!]&&\mbox{on} && \Gamma(t), \\
\Gamma(0) &=\Gamma_0. &&
\end{aligned}
\end{equation}
see \cite[Chapter 1]{PrSi16} for a derivation.
In fact, in \cite{PrSi16} the more general situation where the motion of the fluids is governed by 
the Navier-Stokes equations is also considered.

For later use we note that the
scaled function $p=\pi/\varrho$, with $\pi$ a solution of \eqref{MuT}, satisfies the equivalent problem
\begin{equation}
\label{MuT-scaled}
\begin{aligned}
 \Delta p &=0 &&\mbox{in} && \Omega\setminus \Gamma(t), \\
\partial_\nu p & =0 && \mbox{on} && \partial\Omega, \\
[\![p]\!] &=0 &&\mbox{on} && \Gamma(t), \\\
        p &=\frac{\sigma}{[\![\varrho]\!]} H_\Gamma &&\mbox{on} && \Gamma(t),\\
[\![\varrho]\!]V_\Gamma &= - [\![\varrho^2 k\partial_\nu p ]\!] && \mbox{on} && \Gamma(t), \\
\Gamma(0) &=\Gamma_0. &&
\end{aligned}
\end{equation}

In more generality, one can  also consider the case where $\kappa_i$
depends on the pressure $\pi_i$.
 A variant of Darcy's law is {\em Forchheimer's law}, which reads
$$ g(|u_i|)u_i =-\frac{\kappa(\pi_i)}{\mu_i}\nabla\pi_i,\quad i=1,2,$$
where the function $g$ is strictly positive and $s\mapsto sg(s)$ is strictly increasing. 
Solving this equation for $u_i$ one obtains
\begin{equation}
\label{Forchheimer}
 u_i= -k_i(\pi_i,|\nabla\pi_i|^2)\nabla \pi_i,\quad i=1,2,
\end{equation}
where $k_i$ is strictly positive and satisfies $k_i(p,s)+2s\partial_2k_i(p,s)>0$
for $p\in\R$ and $s\ge 0$. 
These conditions ensure strong ellipticity of the second-oder differential operator
$$-{\rm div}(k_i(\pi_i,|\nabla\pi_i|^2)\nabla\,\cdot) \quad  \mbox{in}\;\; \Omega_i. $$
In case of non-constant densities $\varrho=\varrho(\pi)$,  
the first line in \eqref{Mu} and \eqref{MuT} ought to be replaced by
\begin{equation*}
\label{Verigin}
\partial_t \varrho(\pi) - {\rm div} \big(\varrho(\pi)k(\pi,|\nabla \pi|^2) \nabla\pi\big)=0,
\quad \pi(0)=\pi_0.
\end{equation*}
The resulting model is known as the {\em Verigin problem} (with phase transition in case $j_\Gamma\neq 0$.)
This problem is studied in \cite{PrSi16b}.

It will be shown in Section 3 that problems \eqref{Mu} and \eqref{MuT}
can be cast as a geometric evolution equation 
\begin{equation*}
\label{geo}
V_\Gamma=\sigma G_\Gamma H_\Gamma,\quad t>0,\quad \Gamma(0)=\Gamma_0,
\end{equation*}
where one aims to find a (sufficiently smooth) family of hypersurfaces 
$\Gamma(t)\subset\Omega$ which enclose a domain $\Omega_1(t)$.
Here $G_\Gamma: W^{3/2}_2(\Gamma)\to W^{1/2}_2(\Gamma)$ is linear
and positive semi-definite with respect to the inner product of $L_2(\Gamma)$.

Suppose that the disperse region $\Omega_1$ consists of $m\ge 1$ connected components 
$\Omega_{1,j}$, that is, $\Omega_1=\cup_{j=1}^m \Omega_{1,j}$, while $\Omega_2$ is connected,
see Figure 1. 
Let $\cE$ denote the set of equilibria for \eqref{Mu} and \eqref{MuT}.
\begin{theorem}
\label{conserved}
The Muskat flows \eqref{Mu} and \eqref{MuT}
enjoy the following properties:
\begin{enumerate}
\item[{\bf (a)}] The volumes $|\Omega_{1,j}|$ are preserved for \eqref{Mu},
while the volume $|\Omega_1|$ is preserved for \eqref{MuT}.
\vspace{1mm}
\item[{\bf (b)}] The area functional $|\Gamma|$ is a strict Lyapunov functional
 for \eqref{Mu} and \eqref{MuT}.
\vspace{-3mm} 
\item[{\bf (c)}] The {\rm (}non-degenerate{\rm )} equilibria for \eqref{Mu} consist of $m$ disjoint spheres of arbitrary radii.
$\cE$ is a smooth manifold of dimension $m(n+1)$.
\vspace{1mm}
\item[{\bf (d)}] The {\rm (}non-degenerate{\rm )} equilibria for \eqref{MuT} consist of $m$ disjoint spheres of the same radius.
$\cE$ is a smooth manifold of dimension $mn+1$.
\vspace{1mm}
\item[{\bf (e)}] Each equilibrium $\Gamma_*\in\cE$ is stable for \eqref{Mu}.
\vspace{1mm}
\item[{\bf (f)}] An equilibrium $\Gamma_* \in\cE$ is stable for \eqref{MuT} if $m=1$,
and unstable if $m>1$.
\end{enumerate}
\end{theorem}
\noindent
More precise statements for the assertions in {\bf (e)} and {\bf (f)} are given in Proposition~\ref{pro:normal}
and Theorem~\ref{geom-stability} below. 

It is interesting to note that the {\em Mullins-Sekerka problem}, given by
\begin{equation}
\label{MS}
\begin{aligned}
\Delta \theta &= 0  &&\mbox{in} && \Omega\setminus\Gamma(t), \\
\partial_\nu \theta &= 0 &&\mbox{on} &&\partial\Omega, \\
\lj \theta \rj &=0 &&\mbox{on} && \Gamma(t),\\
\theta &= \sigma H_\Gamma &&\mbox{on} && \Gamma(t),\\
V_\Gamma  & =- [\![d\partial_\nu\theta]\!] &&\mbox{on} && \Gamma(t),\\
\Gamma(0) & = \Gamma_0,  
\end{aligned}
\end{equation}
enjoys the same geometric properties as the Muskat flow with phase transition \eqref{MuT},
see \cite{EsSi98a, PrSi16}.
Problems \eqref{Mu}, \eqref{MuT}, and \eqref{MS}
are all of order 3, and their principal linearizations have equivalent symbols.
Finally, we note that the Mullins-Sekerka problem~\eqref{MS} has the same set of equilibria
as problem \eqref{MuT}, with analogous stability properties.

Problem \eqref{MS} is also known as the quasi-stationary Stefan problem with surface tension
and it describes the motion of a material with phase transition,
with $\theta$ the temperature and $d_i$ the respective (constant) diffusion coefficients.
We refer again to the monograph \cite{PrSi16} for a comprehensive discussion
of the physical background. 

\medskip
The Muskat problem has recently received considerable attention. 
In the case of $\sigma>0$,
the first result on the existence of classical solutions in two dimensions
was obtained by {Hong, Tao and Yi~\cite{HoYi97}.
Regarding the stability of equilibria,
Friedmann and Tao~\cite{FrTa03} proved
stability of a circular steady-state in case that $\Omega_2$ is unbounded.
The authors of [13] state
that the equilibrium is in general not asymptotically stable.

Escher and Matioc~\cite{EsMa11} considered the Muskat problem in a horizontally periodic geometry
with surface tension and gravity included.
Existence and uniqueness of classical  solutions is obtained
and the authors establish exponential stability of certain flat
equilibria. Using bifurcation theory they also identify finger shaped steady-states which
are all unstable.
These results were later refined and extended in
Ehrnstr{\"o}m, Escher, Matioc, Walker~\cite{EEM13, EMM12a, EMW15}.
 Bazaliy and Vasylyeva~\cite{BaVa11} first observed
a waiting time behavior for the two-dimensional Muskat problem
with a non-regular initial surface in the presence of surface tension.

There is an extensive literature for the case of zero surface tension in two dimensions
for vertically superposed fluids.
It is well-known that in this case the problem can be ill-posed.
This situation occurs when the Rayleigh-Taylor condition is not satisfied, that is, 
when the heavier fluid lies above the lighter one, or when the more viscous fluid  pushes the less viscous one.
 Without commenting in more detail we mention the  work of Ambrose~\cite{Amb04},
Escher, Matioc, Walker~\cite{EsMa11,EMM12a,EMW15}, 
Berselli, C{\'o}rdoba, Granero-Belinch{\'o}n~\cite{BCGb14},
Castro, Constantin, C\`ordoba, Fefferman, Gancedo, L{\'o}pez-Fern{\'a}ndez, 
Strain~\cite{CCFG13,CCFG12,CCFGLf11,CCGS13,CCG09,CCG11,CoGa07},
Cheng, Granero-Belinch{\'o}n, Shkoller~\cite{CGbS16}, 
C{\'o}rdoba, Granero-Belinch{\'o}n, Orive-Illera~\cite{CGO14}, 
C{\'o}rdoba, G{\'o}mez-Serrano, Zlato{\v{s}} \cite{CGsZ15a, CGsZ15b},
Constantin, Gancedo, Shvydkoy, Vicol~\cite{CGSV15},
Siegel, Caflisch and Howison~\cite{SCW04}, and Yi~\cite{Yi96,Yi03} 
for various aspects concerning existence of solutions, breakdown of smoothness, finite time turning,
and stability shifting.

The Muskat problem with phase transition \eqref{MuT} 
has been introduced for the first time in \cite{PrSi16}.

Throughout this paper, we use the notation
$B_X(x,r)$ for a ball or radius $r$ and center $x$, with $X$ a normed vector space.
For two given normed vector spaces $X$ and $Y$, $\cB(X,Y)$ denotes the space of all
bounded linear operators from $X$ into $Y$, equipped with the uniform operator norm.
\section{Elliptic transmission problems}
\noindent
In this section we consider an elliptic transmission problem which turns out
to be important for the analysis of the Muskat flow~\eqref{Mu}.

Suppose that $\Omega\subset\RR^n$ is a bounded domain with $C^2$-boundary, consisting of two parts $\Omega_1$ and $\Omega_2$,
as depicted in Figure 1.
Moreover, suppose that $\Gamma=\partial\Omega_1$ is $C^2$,
and $a\in C^1_{ub}(\Omega\setminus \Gamma)$ with $a(x)\ge \alpha>0$ for $x\in \Omega\setminus\Gamma$.
The following elliptic transmission problem, whose formulation is more general than actually needed for this paper,
is also of independent interest.
\begin{equation}
\label{T1-omega}
\begin{aligned}
\omega u- {\rm div}(a\nabla u) &= f   &&\mbox{in} &&  \Omega\setminus\Gamma,\\
\partial_\nu u &=0  &&\mbox{on} && \partial\Omega,\\
[\![u]\!] &= g_1  &&\mbox{on} && \Gamma,\\
[\![a\partial_\nu u]\!] &=g_2  &&\mbox{on} && \Gamma,\\
\end{aligned}
\end{equation}
with $\omega\in\R$.
\begin{proposition}
\label{pro-T1-omega}
Let $1<p<\infty$. Then there exists $\omega_0\in\R$ such that the transmission problem \eqref{T1-omega}
has for each $\omega>\omega_0$ and each
$$(f,g_1,g_2)\in L_p(\Omega)\times W^{2-1/p}_p(\Gamma)\times W^{1-1/p}_p(\Gamma)$$
a unique solution
$u\in W^2_p(\Omega\setminus\Gamma)$.
\end{proposition}
\begin{proof}
Here we give a sketch of the proof, and refer to \cite[Chapter 6]{PrSi16}
for more details. 

\medskip\noindent
{\bf (a)}
We first consider the case with constant coefficients $a_1,a_2$,
flat interface $\Gamma = \RR^{n-1}\times\{0\}=\R^{n-1}$,
and $\Omega_i=\{(x,y)\in \R^{n-1}\times \R: (-1)^iy>0\}$.
Then the problem reads
\begin{equation}
\begin{aligned}
\label{MTMP}
\omega u -a \Delta u &= f, &&y\neq 0,\\
[\![u]\!] &= g_1  &&\mbox{on} \;\; \Gamma,\\
[\![a\partial_\nu u]\!] &=g_2  &&\mbox{on} \;\; \Gamma,\\
\end{aligned}
\end{equation}
with $\nu=e_{n}$ the  outer unit normal of $\Omega_1$.

To obtain solvability of the problem in the right regularity class, 
we transform the problem to the half-space case as follows. Set
\begin{equation*}
\begin{aligned}
\tilde{u}(x,y)&=[u(x,y),u(x,-y)]^{\sf T}, \\
 \tilde{f}(x,y)&=[f(x,y),f(x,-y)]^{\sf T}, 
\end{aligned}
\end{equation*}
for $(x,y)\in\RR^{n-1}\times (0,\infty)$,
and consider the problem
\begin{equation}
\begin{aligned}
\label{TMTMP}
\omega \tilde u -{\rm diag}\,[a_2\Delta, a_1\Delta] \tilde{u}&=\tilde{f} &&\text{in} && \R^n_+,\\
\tilde{u}_2-\tilde{u}_1&=g_1&&\text{on} && \Gamma,\\
a_2\partial_y\tilde{u}_2 + a_1\partial_y\tilde{u}_1 &=g_2 &&\text{on} && \Gamma,\\
\end{aligned}
\end{equation}
where the subscripts $1,2$ refer to the coefficients in the lower resp.\ upper half-space.
Problem \eqref{TMTMP} is strongly elliptic and satisfies the Lopatinskii-Shapiro condition for the half space.
By well-known results for elliptic systems, see for instance \cite[Section 6.3]{PrSi16},
this problem is uniquely solvable in the right class, hence the transmission problem \eqref{MTMP} has this property as well. This proves Proposition \ref{pro-T1-omega} for the constant coefficient case with flat interface.
\medskip\\
{\bf (b)} By perturbation, the result for the flat interface with constant coefficients remains valid for variable coefficients with small deviation from constant ones. By another perturbation argument, 
a proper coordinate transformation transfers the result to the case of a bent interface.
The localization technique finally yields the result for the case of general domains and general coefficients,
see for instance \cite{PrSi16} Section 6.3 for more details.
\end{proof}
The transmission problem
\begin{equation}
\label{T1-A}
\begin{aligned}
\Delta u &= 0   &&\mbox{in} &&  \Omega\setminus\Gamma,\\
\partial_\nu u &=0  &&\mbox{on} && \partial\Omega,\\
\lj u \rj &= h  &&\mbox{on} && \Gamma,\\
\lj k\partial_\nu u \rj &=0  &&\mbox{on} && \Gamma.\\
\end{aligned}
\end{equation}
is closely related to the Muskat problem~\eqref{Mu}.
As in Section 1, $k_i= k|_{\Omega_i}$ is assumed to be constant for $i=1,2.$
We have the following result on solvability.
\begin{proposition} 
\label{pro:T1}
Let $1<p<\infty$.
Then the elliptic transmission problem \eqref{T1-A} has for each $h\in W^{2-1/p}_p(\Gamma)$
a unique solution $u\in W^2_p(\Omega\setminus \Gamma)\cap L_{p,0}(\Omega)$, where
$$L_{p,0}(\Omega)=\{v\in L_p(\Omega): \int_\Omega u\,dx=0\}. $$
\end{proposition}
\goodbreak
\begin{proof}
By Proposition~\ref{pro-T1-omega} we know that problem~\eqref{T1-A}, with the first line
replaced by 
$$\omega_1 u-k\Delta u=0,$$ 
has for each $h\in W^{2-1/p}_p(\Gamma)$ a unique solution $u_1\in W^2_p(\Omega\setminus\Gamma)$,
provided $\omega_1$ is sufficiently large.
In addition, one readily verifies that $u_1\in L_{p,0}(\Omega)$.
Next we show that the problem
\begin{equation}
\label{T1-tilde}
\begin{aligned}
-k\Delta \tilde u &= \omega_1 u_1   &&\mbox{in} &&  \Omega\setminus\Gamma,\\
\partial_\nu \tilde u &=0  &&\mbox{on} && \partial\Omega,\\
\lj \tilde u \rj &= 0  &&\mbox{on} && \Gamma,\\
\lj k\partial_\nu\tilde u \rj &=0  &&\mbox{on} && \Gamma,\\
\end{aligned}
\end{equation}
has a unique solution $\tilde u\in W^2_p(\Omega\setminus\Gamma)\cap L_{p,0}(\Omega)$.
In order to see this, let $X=L_{p,0}(\Omega)$ and consider the linear operator 
${L}:{\sf D}({L})\subset X\to X$ given by
\begin{equation*}
\begin{aligned}
{\sf D}({L}): &=\{v\in W^2_p(\Omega\setminus\Gamma)\cap X: 
\partial_\nu v=0\;\;\mbox{on}\;\;\partial\Omega,
\;\; \lj v \rj = \lj k\partial_\nu v\rj =0\;\;\mbox{on}\;\;\Gamma\}, \\
   {L}v: &= -k\Delta v,\quad v\in {\sf D}({L)}.
\end{aligned}
\end{equation*}
Then ${L}$ has compact resolvent and therefore, its spectrum consists only of eigenvalues of finite 
algebraic multiplicity which, in addition, do not depend on $p$. By a standard energy argument we obtain
$\sigma({L})\subset {\mathbb R}_+$. 
The fact that we restrict ourselves to functions with mean zero
implies that $0$ lies in the resolvent set of ${L}$.
Therefore, \eqref{T1-tilde} has a unique solution $\tilde u \in W^2_p(\Omega\setminus\Gamma)\cap L_{p,0}(\Omega)$.
It is now clear that the function $u=u_1 + \tilde u$
satisfies the assertions of the proposition.
\end{proof}
\section{Volume, area, and equilibria} 
In this section we show that the Muskat problems \eqref{Mu} and \eqref{MuT}
enjoy some important geometric properties, namely conservation of volume and decrease of surface area.
Moreover, we characterize all the equilibria.
We start by showing that both problems \eqref{Mu} and \eqref{MuT-scaled}
can be rewritten as
\begin{equation}
\label{geom}
V_\Gamma=\sigma G_\Gamma H_\Gamma,\quad t>0,\quad \Gamma(0)=\Gamma_0,
\end{equation}
a geometric evolution equation for the motion of $\Gamma(t)$.
Here 
$$G_\Gamma: W^{2-1/q}_q(\Gamma)\to W^{1-1/q}_q(\Gamma),\quad 1<q<\infty, $$ 
is linear and satisfies
\begin{equation}
\label{G-Gamma-properties}
(G_\Gamma g|h)_{L_2(\Gamma)}=(g|G_\Gamma h)_{L_2(\Gamma)},\quad 
(G_\Gamma h|h)_{L_2(\Gamma)}\ge 0, \quad \; g,h\in W^{3/2}_2(\Gamma).
\end{equation}
This can be seen as follows.
Given $h\in W^{3/2}_2(\Gamma)$, let $p\in W^2_2(\Omega\setminus\Gamma)\cap L_{2,0}(\Omega)$ be the 
unique solution of the elliptic transmission problem
\begin{equation}
\label{T1}
\begin{aligned}
\Delta p &= 0   &&\mbox{in} &&  \Omega\setminus\Gamma,\\
\partial_\nu p &=0  &&\mbox{on} && \partial\Omega,\\
[\![p]\!] &= h  &&\mbox{on} && \Gamma,\\
[\![k\partial_\nu p]\!] &=0  &&\mbox{on} && \Gamma,\\
\end{aligned}
\end{equation}
see Proposition~\ref{pro:T1}, and let 
$$G_\Gamma h:=-k\partial_\nu p.$$
It is now clear that \eqref{geom} is equivalent to \eqref{Mu}.
For $g,h\in W^{3/2}_2(\Gamma)$ let $p(g),p(h)\in W^2_2(\Omega\setminus\Gamma)\cap L_{2,0}(\Omega)$ 
be the corresponding solutions of \eqref{T1}.
Then one readily verifies that
\begin{equation}
\label{G-scalar-M}
(G_\Gamma g|h)_{L_2(\Omega)}=(g|G_\Gamma h)_{L_2(\Omega)}=\int_{\Omega} k \nabla p(g)\nabla p(h)\,dx,
\end{equation}
showing that $G_\Gamma$ satisfies \eqref{G-Gamma-properties}.

For the Muskat problem with surface tension~\eqref{MuT-scaled}
we proceed as follows.
 Given $h\in W^{3/2}_2(\Gamma)$, let $p_i\in W^{2}_2(\Omega_i)$ be the
unique solution of the elliptic problem
\begin{equation}
\label{T2-1}
\begin{aligned}
 \Delta p_1 &=0 &&\mbox{in} && \Omega_1, \\
    p_1 &=h &&\mbox{on} &&\Gamma, 
\end{aligned}
\end{equation}
respectively
\begin{equation}
\label{T2-2}
\begin{aligned}
\Delta p_2 &=0 && \mbox{in} && \Omega_2, \\
 \partial_\nu p_2 &=0 && \mbox{on} &&\partial\Omega, \\
p_2 &=h &&\mbox{on} &&\Gamma.
\end{aligned}
\end{equation}
Setting $S_1h:= k_1\partial_\nu p_1$, $S_2h=-k_2\partial_\nu p_2$, and 
$$G_\Gamma h:= \frac{1}{[\![\varrho]\!]^2}(\varrho_1^2 S_1 + \varrho_2^2 S_2)h,$$
we see that the Muskat problem \eqref{MuT-scaled} can be rewritten as \eqref{geom}.
For $g,h\in W^{3/2}_2(\Gamma)$, let $p_i(g), p_i(h)\in W^2_2(\Omega_i)$  be the corresponding solutions
of \eqref{T2-1} and \eqref{T2-2}, respectively. Then one verifies that
\begin{equation}
\label{G-scalar-MT}
(G_\Gamma g|h)_{L_2(\Omega)}=(g|G_\Gamma h)_{L_2(\Omega)}
=\frac{1}{[\![\varrho]\!]^2}\int_\Omega \varrho^2 k\nabla p(g)\cdot\nabla p(h)\,dx,
\end{equation}
and this shows that \eqref{G-Gamma-properties} also holds for \eqref{MuT-scaled}.

Let $\Omega_{1,j}$, $j=1,\ldots,m$, denote the components of $\Omega_1$ and $\Gamma_j$ their boundaries,
and let $\Gamma:=\bigcup_{j=1}^m\Gamma_j$.
Moreover, let ${\sf e}=\chi_{_\Gamma}$ and ${\sf e}_j=\chi_{_{\Gamma_j}}$, where $\chi_{_A}$ denotes
the indicator function of the set $A$.
With \eqref{G-scalar-M} is is not difficult to see that
\begin{equation}
\label{NG-M}
{\sf N}\,(G_\Gamma)={\rm span}\,\{{\sf e}_1,\ldots, {\sf e}_m\}
\end{equation}
for the Muskat problem~\eqref{Mu}, whereas
\begin{equation}
\label{NG-MT}
{\sf N}(G_\Gamma)={\rm span}\,\{{\sf e}\}
\end{equation}
for the Muskat problem with phase transition~\eqref{MuT-scaled}.

\bigskip
\noindent
We are now ready for the {\bf proof of Theorem 1.1(a)-(b)}:

\medskip
\noindent
{\bf (a)}
Let $|\Omega_1(t)|$ denote the volume of $\Omega_1(t)$.
By the change of volume formula, see for instance \cite[Section 2.5]{PrSi16}, 
and \eqref{NG-MT} we obtain
\begin{equation*}
\frac{d}{dt}|\Omega_1(t)|=\int_\Gamma V_\Gamma\,d\Gamma = \sigma\int_{\Gamma}G_\Gamma H_\Gamma\,d\Gamma
=\sigma(H_\Gamma|G_\Gamma {\sf e})_{L_2(\Gamma)}=0.
\end{equation*}
For problem~\eqref{Mu} we obtain by \eqref{NG-M} 
\begin{equation*}
\frac{d}{dt}|\Omega_{1,j}(t)|=\int_{\Gamma_j} V_\Gamma\,d\Gamma = \sigma\int_{\Gamma_j}G_\Gamma H_\Gamma\,d\Gamma
=\sigma(H_\Gamma|G_\Gamma {\sf e}_j)_{L_2(\Gamma)}=0.
\end{equation*}
{\bf (b)-(d)}
Let $|\Gamma(t)|$ denote the surface area of $\Gamma(t)$. 
By the change of area formula, see for instance \cite[Section 2.5]{PrSi16}, 
and~\eqref{G-Gamma-properties} we have
\begin{equation*}
\frac{d}{dt}|\Gamma(t)|= -\int_\Gamma V_\Gamma H_\Gamma \,d\Gamma 
= -\sigma \int_\Gamma (G_\Gamma H_\Gamma)H_\Gamma\,d\Gamma
=-\sigma (G_\Gamma H_\Gamma|H_\Gamma)_{L_2(\Gamma)}\le 0,
\end{equation*}
showing that $|\Gamma|$ is decreasing,
and hence is a Lyapunov function, for \eqref{Mu} and \eqref{MuT-scaled}.
But more is true: $\Phi(\Gamma):=|\Gamma|$ is a strict Lyapunov function.
To see this, suppose that $\frac{d}{dt}\Phi(\Gamma)=0$ for some time $t$. 
Then $(G_\Gamma H_\Gamma|H_\Gamma)=0$.
Let $p$ be the solution of \eqref{T1} with $h= H_\Gamma$;
by~\eqref{G-scalar-M} we obtain
$$
(G_\Gamma h|h)_{L_2(\Omega)}=\int_\Omega k|\nabla p|^2\,d\Gamma=0,
$$
showing that $p$ is constant on $\Omega_2$ and on the connected components $\Omega_{1,j}$ of $\Omega_1$.
Therefore, $h$ is constant on  $\Gamma_j$, that is, 
$h=\sum_{j=1}^m a_j {\sf e}_j$ with some real numbers $a_j$.
This implies that $H_\Gamma$ is constant on each component $\Gamma_j$ of $\Gamma$,
and by Alexandrov's characterization of compact closed hypersurfaces with constant mean curvature,
$\Gamma$ is the union of disjoint spheres, which may all have different radii.
This, in turn, also yields that the equilibria for \eqref{Mu}
consist of disjoint spheres of arbitrary radii.
One shows that $\cE$, the set of all equilibria, is a smooth manifold of dimenion $m(n+1)$,
see for instance~\cite{EsSi98a, PrSi16}.

For the Muskat problem with phase transition we proceed analogously:
let $p_i$, $i=1,2$, be the solution of \eqref{T2-1} and \eqref{T2-2}, respectively. 
Then \eqref{G-scalar-MT} implies that $p_i$ is constant on the connected components of 
$\Omega$. The condition $[\![p]\!]=0$ in turn shows that $p\equiv c$ on $\Omega$,
and this implies that the mean curvature $H_\Gamma$ is constant all over $\Gamma$.
Consequently, $\Gamma$ is the disjoint union of spheres of the same radius
and $\cE$ has dimension $mn+1$.
\hfill{$\square$}
\section{Well-posedness}
In this section, we show that the evolution equation~\eqref{geom}
admits a unique solution which instantaneously regularizes, provided $\Gamma_0\in W^s_p$
with $s>2+(n-1)/p$.

In order to establish this result, we use the common approach of transforming problems
\eqref{Mu} and \eqref{MuT-scaled}, or equivalently problem \eqref{geom},
to a domain with a fixed interface $\Sigma$, where
$\Gamma(t)$ is parameterized over $\Sigma$ by means of a height function $h(t)$.
For this we rely on the {\em Hanzawa transform}, see for instance
\cite[Section 1.3.2]{PrSi16}.

We assume, as before, that $\Omega\subset\R^n$ is a bounded domain with boundary $\partial\Omega$ of class $C^2$,
and that
$\Gamma\subset\Omega$ is a hypersurface of class $C^2$,
i.e.,\ a $C^2$-manifold which is the boundary of a bounded domain
$\Omega_1\subset\Omega$. As above, we set
$\Omega_2=\Omega\backslash\bar{\Omega}_1$, see again Figure 1.
If follows from the results in \cite[Section 2.3.4]{PrSi16}, see also \cite{PrSi13}, that 
$\Gamma$ can be approximated by a real analytic hypersurafce $\Sigma$, in the sense that the Hausdorff distance of the second order
normal bundles is as small as we please. More precisely, given $\eta>0$, there exists an analytic hypersurface
$\Sigma$ such that $d_H(\cN^2\Sigma,\cN^2\Gamma)\leq \eta$. If $\eta>0$ is small enough,
 then $\Sigma$ bounds a domain $\Omega_1^\Sigma$ with $\overline{\Omega^\Sigma_1}\subset\Omega$ and then we set $\Omega^\Sigma_2=\Omega\setminus\overline{\Omega^\Sigma_1}\subset\Omega$.

In the sequel we will freely use the results from \cite[Chapter 2]{PrSi16}.
In particular, we know that the hypersurface $\Sigma$ admits a tubular neighborhood,
which means that there is $a_0>0$ such that the map
\begin{eqnarray*}
&&\Lambda: \Sigma \times (-a_0,a_0)\to \R^n \\
&&\Lambda(p,r):= p+r\nu_\Sigma(p)
\end{eqnarray*}
is a diffeomorphism from $\Sigma \times (-a_0,a_0)$
onto ${\rm im}(\Lambda)$, the image of $\Lambda$. The inverse
$$\Lambda^{-1}:{\rm im}(\Lambda)\to \Sigma\times (-a_0,a_0)$$ of this map
is conveniently decomposed as
$$\Lambda^{-1}(x)=(\Pi_\Sigma(x),d_\Sigma(x)),\quad x\in{\rm im}(\Lambda).$$
Here $\Pi_\Sigma(x)$ means the metric projection of $x$ onto $\Sigma$ and $d_\Sigma(x)$ the signed
distance from $x$ to $\Sigma$; so $|d_\Sigma(x)|={\rm dist}(x,\Sigma)$ and $d_\Sigma(x)<0$ if and only if
$x\in \Omega_1^\Sigma$. In particular we have ${\rm im}(\Lambda)=\{x\in \R^n:\, {\rm dist}(x,\Sigma)<a_0\}$.
The maximal number $a_0$ is given by the radius $r_\Sigma>0$, defined as the largest number $r$ such the exterior and interior ball conditions for $\Sigma$ in $\Omega$ holds.

If ${\rm dist}(\Gamma,\Sigma)$ is small enough,
we may use the map $\Lambda$ to parameterize the unknown free
boundary $\Gamma(t)$ over $\Sigma$ by means of a {height function}
$h(t)$ via
$$\Gamma(t)=\{p+ h(t,p)\nu_\Sigma(p):p\in\Sigma\},\quad t\geq0,$$
for small $t\ge 0$, at least.
We then extend this diffeomorphism to all of $\bar\Omega$ by means of a Hanzawa transform.
With the Weingarten tensor $L_\Sigma$ and the surface gradient
 $\nabla_\Sigma$ we further have
\begin{equation*}
\begin{aligned}
& \nu_\Gamma (h)= \beta(h)(\nu_\Sigma-a(h)),&&  a(h)= M_0(h)\nabla_\Sigma h,\\
& M_0 (h)=(I-hL_\Sigma)^{-1},&& \beta(h) = (1+| a(h)|^2)^{-1/2},
\end{aligned}
\end{equation*}
and
$$V_\Gamma= (\nu_\Sigma\cdot\nu_\Gamma)\partial_t h=\beta(h)\partial_t h.$$
The transformed problem then reads
\begin{equation}\label{geom-trans}
\beta(h)\partial_t h - \sigma G_\Gamma(h)H_\Gamma(h)=0,\quad t>0, \quad h(0)=h_0.
\end{equation}
Recalling the quasilinear structure of $H_\Gamma(h)$ we may apply \cite[Theorem 5.1.1]{PrSi16} 
to the transformed problem. In order to do so, we set
\begin{equation}
\label{spaces}
X_0:=W^{1-1/p}_p(\Sigma),\quad X_1:=W^{4-1/p}_p(\Sigma),\quad X_{\gamma,\mu}:=W^{1+3\mu-4/p}(\Sigma),
\end{equation}
with $\mu\in (1/3+(n+3)/3p,1]$.
Here we note that this choice of $\mu$ implies the embedding
$X_{\gamma,\mu}\hookrightarrow  C^2(\Sigma)$,
showing that the mean curvature $H_\Gamma(h)$ is well-defined.
\begin{theorem}
\label{geom-locex}
Let $p\in(1,\infty)$, and let the spaces $X_0$, $X_1$, and $X_{\gamma,\mu}$ be defined 
as in \eqref{spaces}.

Then \eqref{geom} is locally well-posed in the sense that the transformed problem \eqref{geom-trans}
is locally well-posed for initial values $h_0\in X_{\gamma,\mu}$ which are small in the topology of $C^1(\Sigma)$.
Furthermore, the map $t\mapsto \Gamma(t)$ is real analytic.
\end{theorem}
\begin{proof}
We want to rewrite \eqref{geom-trans} as a quasilinear evolution equation
\begin{equation*}
 \partial_t h +A(h)h = F(h), \quad t>0,\quad h(0)=h_0,
\end{equation*}
where  $h_0$ is small in $C^1(\Sigma)$. 
We recall the representation of the curvature $H_\Gamma$ from \cite[Section 2.2.5]{PrSi16}, 
which reads
$$ H_\Gamma(h) = \beta(h)(c_0(h,\nabla_\Sigma h):\nabla_\Sigma^2h +c_1(h,\nabla_\Sigma h)),$$
where $c_0$ and $c_1$ are real analytic functions, $c_0(0,0)=I$, $c_1(0,0)=H_\Sigma$, and  $-H_\Gamma$ is strongly elliptic 
if $h$ is small in $C^1(\Sigma)$.
Next one shows that the map 
$$ C^2(\Sigma)\to \cB(W^{2-1/p}_p(\Sigma),W^{1-1/p}_p(\Sigma)),\quad h\mapsto G_\Gamma(h),$$
 is real analytic, provided $h$ is small with respect to the topology of $C^1(\Sigma)$.
Furthermore, we write $\beta(h)^{-1}G_\Gamma(h)=G_\Sigma(h)$,
resulting in the problem
\begin{equation}
\label{geom-trans-2}
\partial_t h -\sigma G_\Sigma(h)H_\Gamma(h)=0,\quad t>0, \quad h(0)=h_0.
\end{equation}
Here we note that $G_\Sigma$ is a linear pseudo-differential operator of order 1 on $\Sigma$
for both Muskat problems \eqref{Mu} and \eqref{MuT-scaled}.
We use the decomposition
\begin{equation*}
\begin{aligned} \label{geom ev- AF}
-\sigma G_\Sigma(h)H_\Gamma(h) &= -\sigma G_\Sigma(h) c_0(h,\nabla_\Sigma h):\nabla_\Sigma^2 h
-\sigma G_\Sigma(h)c_1(h,\nabla_\Sigma h)  \\
                    &=:A(h)h - F(h).
\end{aligned}
\end{equation*}
By the techniques developed in \cite[Section 9.5]{PrSi16}, it is not difficult to show that
\begin{equation*}
(A,F): B_{X_{\gamma,\mu}}(0,r)\to \cB(X_1,X_0)\times X_0
\end{equation*} 
is real analytic, provided $r>0$ is small enough. 
Key for this is the embedding $X_{\gamma,\mu}\hookrightarrow C^2(\Sigma)$ which is ensured by the choice of $\mu$. 
It remains to show that $A(h)$ has the property of $L_p$-maximal regularity.

In order to see this, we note that
$$A(0)g= -\sigma G_\Sigma\Delta_\Sigma g, \quad g\in W^{4-1/p}_p(\Sigma),$$
where $\Delta_\Sigma$ is the Laplace-Beltrami operator on $\Sigma$.
It follows from Corollaries 6.6.5 and 6.7.4 in \cite{PrSi16} that
the operator $-A(h)$ with domain ${\sf D}(A(h))=X_1$ has $L_p$-maximal regularity in $X_0$
for both problems~\eqref{Mu} and~\eqref{MuT} for each $h\in B_{X_{\gamma,\mu}}(0,r)$,
provided $r$ is sufficiently small.
Therefore, Theorems~5.1.1 and~5.2.1 in \cite{PrSi16} apply to obtain local well-posedness as well as analyticity in time. 
For analyticity in space we may follow the arguments presented in \cite[Section 9.4]{PrSi16}.
\end{proof}

\section{Stability of equilibria}
Recall that the equilibria of \eqref{Mu} and \eqref{MuT} consist of finitely many  spheres $\Sigma_j:=S(x_j,R_j)$, $1\le j\le m$.
Given such an equilibrium $\Gamma_*=\bigcup_{j=1}^m\Sigma_j$, we choose $\Sigma=\Gamma_*$ as the reference hypersurface.  
The linearization of the transformed problem then reads
\begin{equation}\label{geom-lin}
\partial_t h + \sigma G_\Sigma\cA_\Sigma h=0,
\end{equation}
where
$$\cA_\Sigma\Big|_{\Sigma_j} = -H^\prime_\Gamma(0)\Big|_{\Sigma_j}=-\frac{n-1}{R^2_j}-\Delta_{\Sigma_j},\qquad j=1,\ldots,m,$$
with $R_j$ the radius of the sphere $\Sigma_j$, and $\Delta_{\Sigma_j}$ the Laplace-Beltrami operator of $\Sigma_j$.
This follows from the fact that the Fr\'echet derivative of
$G_\Sigma(h)H_\Gamma(h)$ at $h=0$
(in the direction of $g$) can be evaluated by
\begin{equation*}
\begin{aligned}
\frac{d}{d\ep}\Big|_{\ep=0}G_\Sigma(\ep g)H_\Gamma(\ep g)
=\frac{d}{d\ep}\Big|_{\ep=0}G_\Sigma(\ep g)H_\Gamma(0)+G_\Sigma(0)\frac{d}{d\ep}\Big|_{\ep=0}H_\Gamma(\ep g)
=-G_\Sigma \cA_\Sigma g,
\end{aligned}
\end{equation*}
as $H_\Sigma=H_\Gamma(0)$ is constant on equilibria, and $G_\Sigma(\ep g) {\sf e}=0$.
As the operator $-G_\Sigma\cA_\Sigma$  has maximal regularity, we may apply the stability results 
from~\cite[Chapter 5]{PrSi16}, once we have shown that 0 is normally stable or normally hyperbolic for 
\eqref{geom-trans-2}. 

Before showing the latter we recall the pertinent definitions.
Let
$L:=\sigma G_\Sigma \cA_\Sigma$
be the linearization of $-\sigma G_\Sigma(h)H_\Gamma(h)$ at the equilibrium $h=0$.

Then $0$ is called {\em normally stable} for~\eqref{geom-trans-2}, if
\begin{itemize}
\item[(i)] 
\, near $0$ the set of equilibria $\cF$ is a finite-dimensional $C^1$-manifold in $X_1$,
\item[(ii)] 
\, the tangent space for $\cF$ at $0$ is isomorphic to ${\sf N}(L)$,
\item[(iii)] 
\, $0$ is a semi-simple eigenvalue of $L$, i.e.\ ${\sf R}(L)\oplus {\sf N}(L)=X_0$,
\item[(iv)] 
\, $\sigma(-L)\setminus\{0\}\subset \C_-=\{z\in\C:\, {\rm Re}\, z<0\}$.
\end{itemize}
Moreover, $0$ is {\em normally hyperbolic} if property (iv) is replaced by
\begin{itemize}
\item[(iv$^\prime$)] $\sigma(L)\cap i\R =\{0\}$, $\quad\sigma(-L)\cap \C_+\neq\emptyset$.
\end{itemize}
\noindent
Finally, we say that an equilibrium $\Gamma_*\in\cE$ is normally stable, respectively normally hyperbolic,
for \eqref{geom} 
if $h_*=0$ is normally stable, repspectively normally hyperbolic for the corresponding transformed problem~\eqref{geom-trans-2}
with reference surface $\Sigma =\Gamma_*$. 

\medskip
We are ready to prove the following important result.
\begin{proposition}
\label{pro:normal}
\hfill{ }
\begin{enumerate}
\item[{\bf (i)}] Each equilibrium $\Gamma_*\in\cE$ is normally stable for \eqref{Mu}.
\vspace{1mm}
\item[{\bf (ii)}] An equilibrium $\Gamma_*\in\cE$ is normally stable for \eqref{MuT} if $m=1$,
and normally hyperbolic if $m>1$.
\end{enumerate}
\end{proposition}
\begin{proof}
It follows from our previous considerations that the set of of equilibria form a smooth manifold.
Next we note that $G_\Sigma\cA_\Sigma$ has compact resolvent by boundedness of $\Omega$, 
so we only need to consider its eigenvalues.

\medskip

\noindent
{\bf (a)} We begin with eigenvalue 0. So let $G_\Sigma\cA_\Sigma h=0$. 
Then $\cA_\Sigma h$ belongs to the kernel of $G_\Sigma$, which implies by \eqref{NG-MT} that $\cA_\Sigma h =a {\sf e}$ 
in case \eqref{MuT}, and
$\cA_\Sigma h = \sum_{j=1}^m a_j {\sf e}_j$ in case \eqref{Mu}, see \eqref{NG-M}.
\\
Therefore,  $h=h_0-(R^2/(n-1))a{\sf e}$ for \eqref{MuT}, 
and $h=h_0-\sum_{j=1}^m(R_j^2/(n-1))a_j{\sf e}_j$
in case of \eqref{Mu}, where $h_0\in {\sf N}(\cA_\Sigma)$.
As ${\rm dim}\,{\sf N}(\cA_\Sigma)=mn$, we conclude 
 that the dimension of the kernel ${\sf N}(G_\Sigma\cA_\Sigma)$ equals the dimension of the manifold $\cE$. 
\medskip

\noindent
{\bf (b)} To see that the eigenvalue 0 is semi-simple for $G_\Sigma\cA_\Sigma$,  suppose $(G_\Sigma\cA_\Sigma)^2h=0$. 
Then for \eqref{Mu}
$$ G_\Sigma\cA_\Sigma h = h_0 + \sum_{j=1}^m a_j {\sf e}_j,\quad \mbox{ for some } h_0\in {\sf N}(\cA_\Sigma),\; a_j\in \C.$$
Multiplying this relation with ${\sf e}_l $ in $L_2(\Sigma)$ we obtain $a_j=0$ for all $j$, as $G_\Sigma$ is selfadjoint and 
$G_\Sigma{\sf e}_j=0$. As $\cA_\Sigma$ is also selfadjoint, multiplying with $\cA_\Sigma h$, we obtain
$(G_\Sigma\cA_\Sigma h|\cA_\Sigma h)_{L_2(\Sigma)}=0$, hence $G_\Sigma\cA_\Sigma h=0$. 
The argument for \eqref{MuT} is similar. Consequently, 0 is semi-simple for $G_\Sigma\cA_\Sigma$.

\medskip

\noindent
{\bf (c)} Now suppose that $\lambda\in\C$, $\lambda\neq0$, is an eigenvalue for $-G_\Sigma\cA_\Sigma$, i.e.,
$$ \lambda h + G_\Sigma\cA_\Sigma h=0,$$
for some nontrivial $h$. Taking the inner product with $\cA_\Sigma h$ in $L_2(\Sigma)$ we obtain
$$ \lambda(h|\cA_\Sigma h)_\Sigma +(G_\Sigma\cA_\Sigma h |\cA_\Sigma h)_\Sigma =0.$$
As $G_\Sigma$ and $\cA_\Sigma$ are selfadjoint, this identity implies that $\lambda$ must be real, hence the spectrum of $G_\Sigma\cA_\Sigma$ is real.

We consider now the case \eqref{Mu}; then $(h|{\sf e}_j)_\Sigma=0$ for all $j=1,\ldots,m$. Suppose $\lambda>0$.
As $G_\Sigma$ is positive semi-definite and $\cA_\Sigma$ is so on
the orthogonal complement of ${\rm span }\{{\sf e}_j\}_{j=1}^m$ we see that $(h|\cA_\Sigma h)=0$.  This implies $\cA_\Sigma h=0$ and then $h=0$ as $\lambda>0$. Therefore, there are no nonzero eigenvalues with nonnegative real part, hence in this case  
$0$ is normally stable.

In case \eqref{MuT}, we only obtain $(h|{\sf e})_\Sigma=0$. As $\cA_\Sigma$ is positive semi-definite on functions with mean zero if and only if $\Sigma$ is connected, we may conclude normal stability, provided $\Sigma$ is connected.

\medskip
\noindent
{\bf (d)} Next we show that $G_\Sigma\cA_\Sigma$ has  exactly $(m-1)$ positive eigenvalues in case \eqref{MuT}, 
provided $\Sigma$ has $m$ components $\Sigma_j$. 
In this case we know that $ G_\Sigma$ is positive semi-definite and invertible on $L_{2,0}(\Sigma)$, hence $ G_\Sigma^{-1}$ is positive definite on this space. Therefore, the operator $B_\lambda = \lambda  G_\Sigma^{-1} +\sigma \cA_\Sigma$ has an $(m-1)$-fold negative eigenvalue for $\lambda=0$ and is positive definite for large $\lambda$. This shows that $(m-1)$ eigenvalues must cross the imaginary axis through zero, as $\lambda$ varies from $0$ to $\infty$.
Consequently, $0$ is normally hyperbolic.
\end{proof}

Now we may apply the nonlinear stability results of \cite[Chapter 5]{PrSi16} to obtain the main result of this section.
\goodbreak
\begin{theorem}
\label{geom-stability} 
Let $\Gamma_*$ be an equilibrium of \eqref{geom}
and suppose $s> 2+(n-1)/p$ is fixed.
Then the following assertions hold.
\begin{enumerate}
\item[{\bf (i)}] Problem~\eqref{Mu}{\rm :}\\
$h_*=0$ is stable for \eqref{geom-trans-2} in $W^s_p(\Gamma_*)$.
Any solution $h$ starting close to $h_*=0$ in $W^s_p(\Gamma_*)$ exists globally and converges to an 
equilibrium $h_\infty$ of \eqref{geom-trans-2} in $W^s_p(\Gamma_*)$
at an exponential rate.
\vspace{1mm}
\item[{\bf (ii)}]  Problem~\eqref{MuT}{\rm :}\\
$h_*=0$ is stable for \eqref{geom-trans-2} in $W^s_p(\Gamma_*)$, 
provided $\Gamma_*$ is connected. In this case, the same assertions as in {\bf (i)} hold.
\\
If $\Gamma_*$ is disconnected, then $h_*=0$ is unstable in $W^s_p(\Gamma_*)$. 
A solution $h$ starting close to $h_*=0$ and staying close to the set of equilibria in the topology of $W^s_p(\Gamma_*)$ exists globally and converges to some equilibrium $h_\infty$ of \eqref{geom-trans-2} in $W^s_p(\Gamma_*)$ at an exponential rate.
\end{enumerate}
In both cases, $h_\infty$ corresponds to some $\Gamma_\infty\in\cE$.
\end{theorem}
\noindent
{\bf Proof of Theorem 1.1(e)-(f)}:
The assertions follow  from Theorem~\ref{geom-stability} by means of the transformation alluded to 
at the beginning of Section 4.
\hfill{$\square$}

\bigskip
\noindent
So in conclusion, the Muskat flow with phase transition sees the phenomenon of Ostwald-ripening, 
while the Muskat flow does not share this property. Physically speaking, \eqref{MuT} is spatially non-local 
so that different parts of the surface {see each other}. On the other hand, \eqref{Mu} is also non-local in space, 
but the coupling between different parts of the surface is not strong enough to enable Ostwald-ripening.
\goodbreak
\section{Semiflow and long-time behavior}
It can be shown that the closed $C^2$-hypersurfaces contained in $\Omega$ which bound a region
$\Omega_1\subset\subset\Omega$ form a $C^2$-manifold, denoted by $\cMH^2(\Omega)$,
see for instance \cite{PrSi13} or \cite[Chapter 2]{PrSi16}.
The charts are the normal parameterizations over a reference hypersurface $\Sigma$, and the
tangent space consists of the normal vector fields of $\Sigma$.

We define the state manifold of \eqref{geom} by means of
\begin{equation}\label{geom-statemanif}
\cSM^s(\Omega):=\{\Gamma\in\cMH^2(\Omega)\st \Gamma\in W^{s}_p\},
\quad s>2 + (n-1)/p.
\end{equation}
The topology of $\cSM^s(\Omega)$ is that induced
by the canonical level functions $\varphi_\Gamma$ in $W^{s}_p(\Omega)$,
see \cite[Section 2.4.2]{PrSi16}. 
By Theorem  \ref{geom-locex} we see that given an initial surface $\Gamma_0\in \cSM^s(\Omega)$ we find $a>0$ and $\Gamma:[0,a]\to\cSM^s(\Omega)$ continuous such that $\Gamma(0)=\Gamma_0$ and $\Gamma(\cdot)$ is an $L_p$-solution in the sense that $\Gamma$ is obtained as the push forward of the solution of the transformed problem \eqref{geom-trans}.
We may extend such an orbit in $\cSM^s(\Omega)$ to a maximal time interval $J(\Gamma_0):=[0,t_+(\Gamma_0))$.
Basically there are two facts which prevent the solution from being global, namely
    \begin{itemize}
    \item {\em Regularity}: the norm of $\Gamma(t)$ in $W_p^{s}$ may become unbounded as $t\to t_+(\Gamma_0)$;
    \vspace{-3mm}
    \item {\em Geometry}: the topology of the interface $\Gamma(t)$ may change,
    or the interface may touch the boundary of $\Omega$, or part of it may shrink to points.
    \end{itemize}
We say that the solution $\Gamma(t)$ satisfies a\emph{uniform ball condition},
if there is a number $r>0$ such that for each $t\in J_0:=[0,t_+(\Gamma_0))$ and each $p\in \Gamma(t)$
there are balls $B(x_i,r)\subset \Omega_i$, $i=1,2$, such that
$\bar B(x_i,r)\cap \Gamma(t)=\{p\}$.
 The main result of this section reads as follows.

\begin{theorem}\label{geom-global}
Let  $\Gamma(t)$ be a solution of the geometric evolution equation \eqref{geom} on its maximal time interval $J(\Gamma_0)$. Assume furthermore that
\begin{enumerate}
    \item[{\bf (i)}] $|\Gamma(t)|_{W_p^s}\le M<\infty$ for all $t\in J(\Gamma_0)$, and
    \vspace{1mm}
    \item[{\bf (ii)}] $\Gamma(t)$ satisfies a uniform ball condition.
\end{enumerate}
Then $J(\Gamma_0)=\R_+$, i.e.,\ the solution exists globally, and $\Gamma(t)$ converges in $\cSM^s$ to an equilibrium $\Gamma_\infty\in\cE$ at an exponential rate.
The converse is also true: if a global solution converges in $\cSM^s$ to an equilibrium, then (i)  and (ii) are valid.
\end{theorem}
\begin{proof}
The proof relies on \cite[Theorem 4.3]{KPW10} and follows the same lines as that of
Theorem~5.2 in \cite{KPW10}; see also \cite{PrSi16}, Theorems 5.7.2 and 11.4.1.
\end{proof}

\bigskip



\begin{thebibliography}{99}

\bibitem{Amb04}
D.~M. Ambrose.
\newblock Well-posedness of two-phase {H}ele-{S}haw flow without surface
  tension.
\newblock {\em European J. Appl. Math.}, 15(5):597--607, 2004.

\bibitem{BaVa11}
B.~V. Bazaliy and N.~Vasylyeva.
\newblock The {M}uskat problem with surface tension and a nonregular initial
  interface.
\newblock {\em Nonlinear Anal.}, 74(17):6074--6096, 2011.

\bibitem{BCGb14}
L.~C. Berselli, D.~C{\'o}rdoba, and R.~Granero-Belinch{\'o}n.
\newblock Local solvability and turning for the inhomogeneous {M}uskat problem.
\newblock {\em Interfaces Free Bound.}, 16(2):175--213, 2014.

\bibitem{CCFG13}
{\'A}.~Castro, D.~C{\'o}rdoba, C.~Fefferman, and F.~Gancedo.
\newblock Breakdown of smoothness for the {M}uskat problem.
\newblock {\em Arch. Ration. Mech. Anal.}, 208(3):805--909, 2013.

\bibitem{CCFG12}
{\'A}.~Castro, D.~C{\'o}rdoba, C.~Fefferman, F.~Gancedo, and
  M.~L{\'o}pez-Fern{\'a}ndez.
\newblock Rayleigh-{T}aylor breakdown for the {M}uskat problem with
  applications to water waves.
\newblock {\em Ann. of Math. (2)}, 175(2):909--948, 2012.

\bibitem{CCFGLf11}
A.~Castro, D.~C{\'o}rdoba, C.~L. Fefferman, F.~Gancedo, and
  M.~L{\'o}pez-Fern{\'a}ndez.
\newblock Turning waves and breakdown for incompressible flows.
\newblock {\em Proc. Natl. Acad. Sci. USA}, 108(12):4754--4759, 2011.

\bibitem{CGbS16}
C.~H.~A. Cheng, R.~Granero-Belinch{\'o}n, and S.~Shkoller.
\newblock Well-posedness of the {M}uskat problem with {$H^2$} initial data.
\newblock {\em Adv. Math.}, 286:32--104, 2016.

\bibitem{CCGS13}
P.~Constantin, D.~C{\'o}rdoba, F.~Gancedo, and R.~M. Strain.
\newblock On the global existence for the {M}uskat problem.
\newblock {\em J. Eur. Math. Soc. (JEMS)}, 15(1):201--227, 2013.

\bibitem{CGSV15}
P.~Constantin, G.~Francisco, S.~Roman, and V.~Vicol.
\newblock Global regularity for the 2{D} {M}uskat equations with finite slope.
\newblock {\em arXiv:1507.01386}, 2015.

\bibitem{CCG09}
A.~Cordoba, D.~Cordoba, and F.~Gancedo.
\newblock The {R}ayleigh-{T}aylor condition for the evolution of irrotational
  fluid interfaces.
\newblock {\em Proc. Natl. Acad. Sci. USA}, 106(27):10955--10959, 2009.

\bibitem{CCG11}
A.~C{\'o}rdoba, D.~C{\'o}rdoba, and F.~Gancedo.
\newblock Interface evolution: the {H}ele-{S}haw and {M}uskat problems.
\newblock {\em Ann. of Math. (2)}, 173(1):477--542, 2011.

\bibitem{CoGa07}
D.~C{\'o}rdoba and F.~Gancedo.
\newblock Contour dynamics of incompressible 3-{D} fluids in a porous medium
  with different densities.
\newblock {\em Comm. Math. Phys.}, 273(2):445--471, 2007.

\bibitem{CGsZ15a}
D.~C{\'o}rdoba, J.~G{\'o}mez-Serrano, and A.~Zlato{\v{s}}.
\newblock A note on stability shifting for the {M}uskat problem.
\newblock {\em Philos. Trans. A}, 373(2050):20140278, 10, 2015.

\bibitem{CGsZ15b}
D.~C{\'o}rdoba, J.~G{\'o}mez-Serrano, and A.~Zlato{\v{s}}.
\newblock A note on stability shifting for the {M}uskat problem {II}: stable to
  unstable and back to stable.
\newblock {\em arXiv:1512.02564}, 2015.

\bibitem{CGO14}
D.~C{\'o}rdoba~Gazolaz, R.~Granero-Belinch{\'o}n, and R.~Orive-Illera.
\newblock The confined {M}uskat problem: differences with the deep water
  regime.
\newblock {\em Commun. Math. Sci.}, 12(3):423--455, 2014.

\bibitem{EEM13}
M.~Ehrnstr{\"o}m, J.~Escher, and B.-V. Matioc.
\newblock Steady-state fingering patterns for a periodic {M}uskat problem.
\newblock {\em Methods Appl. Anal.}, 20(1):33--46, 2013.

\bibitem{EMM12a}
J.~Escher, A.-V. Matioc, and B.-V. Matioc.
\newblock A generalized {R}ayleigh-{T}aylor condition for the {M}uskat problem.
\newblock {\em Nonlinearity}, 25(1):73--92, 2012.

\bibitem{EsMa11}
J.~Escher and B.-V. Matioc.
\newblock On the parabolicity of the {M}uskat problem: well-posedness,
  fingering, and stability results.
\newblock {\em Z. Anal. Anwend.}, 30(2):193--218, 2011.

\bibitem{EMW15}
J.~Escher, B.-V. Matioc, and C.~Walker.
\newblock The domain of parabolicity for the {M}uskat problem.
\newblock {\em arXiv:1507.02601}, 2015.

\bibitem{EsSi98a}
J.~Escher and G.~Simonett.
\newblock A center manifold analysis for the {M}ullins-{S}ekerka model.
\newblock {\em J. Differential Equations}, 143(2):267--292, 1998.

\bibitem{FrTa03}
A.~Friedman and Y.~Tao.
\newblock Nonlinear stability of the {M}uskat problem with capillary pressure
  at the free boundary.
\newblock {\em Nonlinear Anal.}, 53(1):45--80, 2003.

\bibitem{HoYi97}
J.~Hong, Y.~Tao, and F.~Yi.
\newblock Muskat problem with surface tension.
\newblock {\em J. Partial Differential Equations}, 10(3):213--231, 1997.

\bibitem{KPW10}
M.~K{\"o}hne, J.~Pr{\"u}ss, and M.~Wilke.
\newblock On quasilinear parabolic evolution equations in weighted
  {$L_p$}-spaces.
\newblock {\em J. Evol. Equ.}, 10(2):443--463, 2010.

\bibitem{Mus34}
M.~Muskat.
\newblock Two fluid systems in porous media. {T}he encroachment of water into
  an oil sand.
\newblock {\em Physics}, 5:250--264, 1934.

\bibitem{MuWy37}
M.~Muskat and R.~D. Wyckoff.
\newblock {\em The flow of homogeneous fluids through porous media}.
\newblock McGraw-Hill, New York, London, 1937.

\bibitem{PrSi13}
J.~Pr{\"u}ss and G.~Simonett.
\newblock On the manifold of closed hypersurfaces in {${\mathbb R}^n$}.
\newblock {\em Discrete Cont. Dyn. Sys. A}, 33:5407--5428, 2013.

\bibitem{PrSi16}
J.~Pr\"uss and G.~Simonett.
\newblock {\em Moving interfaces and quasilinear parabolic evolution
  equations}, volume 105 of {\em Monographs in Mathematics}.
\newblock Birkh\"auser, 2016.


\bibitem{PrSi16b}
J.~Pr{\"u}ss and G.~Simonett.
\newblock The {V}erigin problem with and without phase transition.
\newblock 2016.
\newblock Submitted.

\bibitem{SCW04}
M.~Siegel, R.~E. Caflisch, and S.~Howison.
\newblock Global existence, singular solutions, and ill-posedness for the
  {M}uskat problem.
\newblock {\em Comm. Pure Appl. Math.}, 57(10):1374--1411, 2004.

\bibitem{Yi96}
F.~Yi.
\newblock Local classical solution of {M}uskat free boundary problem.
\newblock {\em J. Partial Differential Equations}, 9(1):84--96, 1996.

\bibitem{Yi03}
F.~Yi.
\newblock Global classical solution of {M}uskat free boundary problem.
\newblock {\em J. Math. Anal. Appl.}, 288(2):442--461, 2003.

\end{thebibliography}
\end{document}